\documentclass[12pt]{amsart}

\usepackage{amsmath}
\usepackage{amscd}
\usepackage{amssymb}
\usepackage{amsfonts}
\usepackage{amsthm}
\usepackage{enumerate}
\usepackage{hyperref}
\usepackage{color}
\usepackage{graphicx}
\theoremstyle{plain}
\newtheorem{thm}{Theorem}[section]

\newtheorem{prop}[thm]{Proposition}
\newtheorem{lem}[thm]{Lemma}
\newtheorem{cor}[thm]{Corollary}

\newtheorem{conj}[thm]{Conjecture}

\theoremstyle{definition}

\newtheorem{exmp}[thm]{Example}

\newtheorem{dfns-rems}[thm]{Definitions and Remarks}
\newtheorem{notas-rems}[thm]{Notations and Remarks}
\newtheorem{exmps-rems}[thm]{Examples and Remarks}

\begin{document}

% ------------------------------------------------------------------------

\title[Regularity of small symbolic powers of edge ideals]{On the regularity of small symbolic powers of edge ideals of graphs}

% ------------------------------------------------------------------------

\author[S. A. Seyed Fakhari]{S. A. Seyed Fakhari}

\address{S. A. Seyed Fakhari, School of Mathematics, Statistics and Computer Science,
College of Science, University of Tehran, Tehran, Iran, and Institute of Mathematics, Vietnam Academy of Science and Technology, 18 Hoang Quoc Viet, Hanoi, Vietnam.}

\email{aminfakhari@ut.ac.ir}

% ------------------------------------------------------------------------

\begin{abstract}
Assume that $G$ is a graph with edge ideal $I(G)$ and let $I(G)^{(s)}$ denote the $s$-th symbolic power of $I(G)$. It is proved that for every integer $s\geq 1$, $${\rm reg}(I(G)^{(s+1)})\leq \max\bigg\{{\rm reg}(I(G))+2s, {\rm reg}\big(I(G)^{(s+1)}+I(G)^s\big)\bigg\}.$$As a consequence, we conclude that ${\rm reg}(I(G)^{(2)})\leq {\rm reg}(I(G))+2$, and ${\rm reg}(I(G)^{(3)})\leq {\rm reg}(I(G))+4$. Moreover, it is shown that if for some integer $k\geq 1$, the graph $G$ has no odd cycle of length at most $2k-1$, then ${\rm reg}(I(G)^{(s)})\leq 2s+{\rm reg}(I(G))-2$, for every integer $s\leq k+1$. Finally, it is proven that ${\rm reg}(I(G)^{(s)})=2s$, for $s\in \{2, 3, 4\}$, provided that the complementary graph $\overline{G}$ is chordal.
\end{abstract}

% ------------------------------------------------------------------------


% ------------------------------------------------------------------------


% ------------------------------------------------------------------------

\thanks{This research is partially funded by the Simons Foundation Grant Targeted for Institute of Mathematics, Vietnam Academy of Science and Technology.}

% ------------------------------------------------------------------------

\maketitle

%%%%%%%%%%%%%%%%%%%%%%%%%%%%%%%%%%%%%%%%%%%%%%%%%%%%%%%%%%%%%%%%%%%%%%%%%%

\section{Introduction} \label{sec1}

Let $\mathbb{K}$ be a field and $S = \mathbb{K}[x_1,\ldots,x_n]$  be the
polynomial ring in $n$ variables over $\mathbb{K}$. Suppose that $M$ is a graded $S$-module with minimal free resolution
$$0  \longrightarrow \cdots \longrightarrow  \bigoplus_{j}S(-j)^{\beta_{1,j}(M)} \longrightarrow \bigoplus_{j}S(-j)^{\beta_{0,j}(M)}   \longrightarrow  M \longrightarrow 0.$$
The Castelnuovo--Mumford regularity (or simply, regularity) of $M$, denoted by ${\rm reg}(M)$, is defined as
$${\rm reg}(M)=\max\{j-i|\ \beta_{i,j}(M)\neq0\},$$
and it is an important invariant in commutative algebra and algebraic geometry.

Cutkosky, Herzog, Trung, \cite{cht}, and independently  Kodiyalam \cite{ko}, proved that for a homogenous ideal $I$ in the polynomial ring, the function ${\rm reg}(I^s)$ is linear, for $s\gg0$, i.e., there exist integers $a(I)$ and $b(I)$ such that $${\rm reg} (I^s)=a(I)s+b(I) \ \ \ \ {\rm for} \ s\gg 0.$$
It is known that $a(I)$ is bounded above by the maximum degree of elements in a minimal generating set of $I$. But a general combinatorial bound for $b(I)$ is not known, even if $I$ is monomial ideal. However, Alilooee, Banerjee, Beyarslan and H${\rm \grave{a}}$ \cite[Cojecture 1]{bbh1} conjectured that $b(I)\leq {\rm reg}(I)-2$, when $I=I(G)$ is the edge ideal of a graph. Indeed, they proposed the following stronger conjecture.

\begin{conj} [\cite{bbh1}, Conjecture 1] \label{conj1}
For every graph $G$ and any integer $s\geq 1$, we have$${\rm reg}(I(G)^s)\leq 2s+{\rm reg}(I(G))-2.$$
\end{conj}

It is obvious that for any graph $G$, the inequality ${\rm reg}(I(G))\geq 2$ holds. The characterization of graphs with ${\rm reg}(I(G))=2$ is obtained by Fr${\rm \ddot{o}}$berg \cite[Theorem 1]{f}. In fact, he proved that ${\rm reg}(I(G))=2$ if and only if the complementary graph $\overline{G}$ is chordal, i.e., has no induced cycle of length at least four. Herzog, Hibi and Zheng \cite{hhz} showed that for this class of graphs, indeed we have ${\rm reg}(I(G)^s)=2s$, for every integer $s\geq 1$. In other words, Conjecture \ref{conj1} is true for any graph which has a chordal complement. Conjecture \ref{conj1} is also known to be true for other classes of graphs, including

$\bullet$ gap-free and cricket-free graphs \cite{b},

$\bullet$ cycles and unicyclic graphs \cite{abs, bht},

$\bullet$ very well-covered graphs \cite{js},

$\bullet$ bipartite graphs \cite{bn},

$\bullet$ gap-free and diamond-free graphs \cite{e1},

$\bullet$ gap-free and $C_4$-free graphs \cite{e2},

$\bullet$ Cameron-Walker graphs \cite{bbh},

$\bullet$ subclasses of bicyclic graphs \cite{cjnp, g}.\\

It is also reasonable to study the regularity of symbolic powers of edge ideals. In this direction, Minh posed the following conjecture (see \cite{ghos}).

\begin{conj} \label{conj2}
For every graph $G$ and any integer $s\geq 1$, we have$${\rm reg}(I(G)^s)={\rm reg}(I(G)^{(s)}).$$
\end{conj}

In the above conjecture, $I(G)^{(s)}$, denotes the $s$-th symbolic powers of $I(G)$.

If $G$ is a bipartite graph, then we know from \cite[Theorem 5.9]{svv}, that $I(G)^{(s)}=I(G)^s$, for every integer $s\geq 1$. Therefore, Conjecture \ref{conj2} is trivially true for this class of graphs. On the other, it is known that Conjecture \ref{conj2} is also true for cycles, unicyclic graphs and Cameron-Walker graphs (see \cite{ghos, s8, s9}, respectively). Moreover, Jayanthan and Kumar \cite{jk} proved Conjecture \ref{conj2} for some classes of graphs which are obtained by the clique sum of odd cycles and bipartite graphs.

The concentration of this paper is on the following conjecture which is obtained by combination of Conjectures \ref{conj1} and \ref{conj2}.

\begin{conj} \label{conj3}
For every graph $G$ and any integer $s\geq 1$, we have$${\rm reg}(I(G)^{(s)})\leq 2s+{\rm reg}(I(G))-2.$$
\end{conj}

It follows from the above mentioned results that Conjecture \ref{conj3} is true for bipartite graphs, cycles, unicyclic graphs and Cameron-Walker graphs. Moreover, in \cite{s11}, we proved that for every chordal graph $G$ and every integer $s\geq 1$, the equality$${\rm reg}(I(G)^{(s)})=2s+{\rm reg}(I(G))-2$$holds. Hence, Conjecture \ref{conj3} is also true for chordal graphs.

Recently, Banerjee and Nevo \cite{bn} proved that for every graph $G$, we have ${\rm reg}(I(G)^2)\leq {\rm reg}(I(G))+2$. Thus, confirming Conjecture \ref{conj1} in the case of $s=2$. Banerjee, Beyarslan and H${\rm \grave{a}}$ \cite[Theorem 5.3]{bbh} proved the same inequality for the second symbolic power, under the extra assumption that ${\rm reg}(I(G):x)\leq {\rm reg}(IG))-1$. In Corollary \ref{twth}, we prove  that the inequality ${\rm reg}(I(G)^{(2)})\leq {\rm reg}(I(G))+2$ is true for any arbitrary graph (thus, generalizing \cite[Theorem 5.3]{bbh}). In other words, Conjecture \ref{conj3} is true for $s=2$. In the same corollary, we will see that Conjecture \ref{conj3} is true for $s=3$, too. In order to prove Corollary \ref{twth}, we first show in Theorem \ref{1main} that for every graph $G$ and any integer $s\geq 1$, we have$${\rm reg}(I(G)^{(s+1)})\leq \max\bigg\{{\rm reg}(I(G))+2s, {\rm reg}\big(I(G)^{(s+1)}+I(G)^s\big)\bigg\}.$$Hence, to verify Conjecture \ref{conj3}, we need to appropriately bound the regularity of the ideals in the form $I(G)^{(s+1)}+I(G)^s$. It is clear that for any graph $G$, $I(G)^{(2)}$ is contained in $I(G)$. Also, it is not difficult to check that $I(G)^{(3)}\subseteq I(G)^2$. Therefore, Corollary \ref{twth} easily follows from Theorem \ref{1main} and the above mentioned result of Banerjee and Nevo \cite{bn}. Indeed, we prove something more. It is shown in Proposition \ref{cont} that if $G$ has no odd cycle of length at most $2s-3$, then $I(G)^{(s+1)}\subseteq I(G)^s$. Thus, using Theorem \ref{1main}, we conclude that for such a graph, we have$${\rm reg}(I(G)^{(s+1)})\leq \max\big\{{\rm reg}(I(G))+2s, {\rm reg}(I(G)^s)\big\},$$(see Theorem \ref{2main}). As a consequence, we obtain in Corollary \ref{resycy} that if $G$ has no odd cycle of length at most $2k-1$, then for every integer $s\leq k+1$,$${\rm reg}(I(G)^{(s)})\leq 2s+{\rm reg}(I(G))-2.$$Morever, it is shown that for any graph with this property,$${\rm reg}(I(G)^s)\leq 2s+{\rm reg}(I(G))-2,$$for every $s\leq k$, Corollary \ref{regord}. This extends a recent result of Banerjee and Nevo \cite[Theorem 1.1(ii)]{bn}, which states that for every bipartite graph $G$ and every integer $s\geq 1$, we have$${\rm reg}(I(G)^s)\leq 2s+{\rm reg}(I(G))-2.$$

In Section \ref{sec4}, we study the regularity of the symbolic powers of edge ideals of graphs which have chordal complements. As we mentioned above, Herzog, Hibi and Zheng \cite{hhz} proved that for every graph with this property, ${\rm reg}(I(G)^s)=2s$. However, it is not so much known about the regularity of symbolic powers of edge ideals of these graphs. It immediately follows from Corollary \ref{twth} that for every graph $G$ with chordal complement, ${\rm reg}(I(G)^{(s)})=2s$, for $s\in \{2,3\}$. In Theorem \ref{fococh}, we show that the same equality holds for $s=4$, too.

%%%%%%%%%%%%%%%%%%%%%%%%%%%%%%%%%%%%%%%%%%%%%%%%%%%%%%%%%%%%%%%%%%%%%%%%%%

\section{Preliminaries} \label{sec2}

In this section, we provide the definitions and basic facts which will be used in the next sections.

Let $G$ be a simple graph with vertex set $V(G)=\big\{x_1, \ldots,
x_n\big\}$ and edge set $E(G)$. We identify the vertices (resp. edges) of $G$ with variables (resp. corresponding quadratic monomials) of $S$. For a vertex $x_i$, the {\it neighbor set} of $x_i$ is $N_G(x_i)=\{x_j\mid x_ix_j\in E(G)\}$. The {\it complementary graph}
$\overline{G}$ is the graph with vertex set $V(\overline{G})=V(G)$ and its edge set $E(\overline{G})$ consists of those $2$-element subsets of $V(G)$ which do not belong to $E(G)$. For every subset $U\subset V(G)$, the graph $G\setminus U$ has vertex set $V(G\setminus U)=V(G)\setminus U$ and edge set $E(G\setminus U)=\{e\in E(G)\mid e\cap U=\emptyset\}$. A subgraph $H$ of $G$ is called {\it induced} provided that two vertices of $H$ are adjacent if and only if they are adjacent in $G$. The graph $G$ is {\it chordal} if it has no induced cycle of length at least four. Two disjoint edges $e, e'\in E(G)$ are said to be a {\it gap}, if they form an induced subgraph of $G$. The graph $G$ is {\it gap-free} if it has no gap. The complete graph on $n$ vertices is denote by $K_n$. The graph $K_3$ is also called a {\it triangle}. A subset $A\subseteq V(G)$ is an {\it independent set} in $G$ if no pair of vertices in $A$ are adjacent. A subset $C$ of $V(G)$ is a {\it vertex cover} of $G$ if every edge of $G$ is incident to at least one vertex of $C$. A vertex cover $C$ is a {\it minimal vertex cover} if no proper subset of $C$ is a vertex cover of $G$. The set of minimal vertex covers of $G$ will be denoted by $\mathcal{C}(G)$.

Let $G$ be a graph with vertex set $V(G)=\big\{x_1, \ldots, x_n\big\}$ and edge set $E(G)$. The {\it edge ideal} of $G$, denoted by $I(G)$, is the monomial ideal of $S$ which is generated by quadratic squarefree monomials corresponding to edges of $G$, i.e.,$$I(G)=\big(x_ix_j: x_ix_j\in E(G)\big).$$Computing and finding bounds for the regularity of edge ideals and their powers have been studied by a number of researchers (see for example \cite{ab},  \cite{abs}, \cite{b},  \cite{bbh}, \cite{bht}, \cite{dhs}, \cite{ha}, \cite{jns}, \cite{k}, \cite{msy}, \cite{sy} and \cite{wo}). We refer the reader to \cite{bbh1} for a survey on this topic.

For every monomial $u\in S$ and for every variable $x_i$, the {\it degree of $u$ with respect to $x_i$}, denoted by ${\rm deg}_{x_i}(u)$ is the largest power of $x_i$ which divides $u$. For a monomial ideal $I$, we denote by $G(I)$, the set of minimal monomial generators of $I$. For every subset $A$ of $\big\{x_1, \ldots, x_n\big\}$, $\mathfrak{p}_A$ denotes the monomial prime ideal which is generated by the variables belonging to $A$. It is well-known that for every graph $G$ with edge ideal $I(G)$,$$I(G)=\bigcap_{C\in \mathcal{C}(G)}\mathfrak{p}_C.$$

We close this section by recalling the definition of symbolic powers.

Let $I$ be an ideal of $S$ and let ${\rm Min}(I)$ denote the set of minimal primes of $I$. For every integer $s\geq 1$, the $s$-th {\it symbolic power} of $I$,
denoted by $I^{(s)}$, is defined to be$$I^{(s)}=\bigcap_{\frak{p}\in {\rm Min}(I)} {\rm Ker}(S\rightarrow (S/I^s)_{\frak{p}}).$$Assume that $I$ is a squarefree monomial ideal in $S$ and suppose $I$ has the irredundant
primary decomposition $$I=\frak{p}_1\cap\ldots\cap\frak{p}_r,$$ where every
$\frak{p}_i$ is an ideal generated by a subset of the variables of
$S$. It follows from \cite[Proposition 1.4.4]{hh} that for every integer $s\geq 1$, $$I^{(s)}=\frak{p}_1^s\cap\ldots\cap
\frak{p}_r^s.$$In particular, for every graph $G$,$$I(G)^{(s)}=\bigcap_{C\in \mathcal{C}(G)}\mathfrak{p}_C^s.$$

%%%%%%%%%%%%%%%%%%%%%%%%%%%%%%%%%%%%%%%%%%%%%%%%%%%%%%%%%%%%%%%%%%%%%%%%%%

\section{Upper bounds for the regularity of symbolic powers} \label{sec3}

The aim of this section is to prove Conjecture \ref{conj3} in the following cases: (i) For $s=2,3$. (ii) For any integer $s\leq k+1$, where $k\geq 1$ is an integer with the property that $G$ has no odd cycle of length at most $2k-1$. To achieve this goal, in Theorem \ref{1main}, we determine an upper bound for the regularity of symbolic powers of edge ideals. The following lemma is the first step in the proof of Theorem \ref{1main} and its proof is based on an appropriate ordering of minimal monomial generators of ordinary powers of edge ideals, which has been provided by Banerjee \cite{b}.

\begin{lem} \label{rfirst}
Assume that $G$ is a graph and $s\geq 1$ is a positive integer. Let $G(I(G)^s)=\{u_1, \ldots, u_m\}$ denote the set of minimal monomial generators of $I(G)^s$. Then$${\rm reg}(I(G)^{(s+1)})\leq \max\bigg\{{\rm reg}\big(I(G)^{(s+1)}:u_i\big)+2s, 1\leq i\leq m, {\rm reg}\big(I(G)^{(s+1)}+I(G)^s\big)\bigg\}.$$
\end{lem}

\begin{proof}
Using \cite[Theorem 4.12]{b}, we may assume that for every pair of integers $1\leq j< i\leq m$, one of the following conditions hold.
\begin{itemize}
\item [(i)] $(u_j:u_i) \subseteq (I(G)^{s+1}:u_i)\subseteq (I(G)^{(s+1)}:u_i)$; or
\item [(ii)] there exists an integer $k\leq i-1$ such that $(u_k:u_i)$ is generated by a subset of variables, and $(u_j:u_i)\subseteq (u_k:u_i)$.
\end{itemize}
Thus, for every integer $i\geq 2$,
\begin{align*}
\big((I(G)^{(s+1)}, u_1, \ldots, u_{i-1}):u_i\big)=(I(G)^{(s+1)}:u_i)+({\rm some \ variables}).
\end{align*}
Hence, we conclude from \cite[Lemma 2.10]{b} that
\[
\begin{array}{rl}
{\rm reg}\big((I(G)^{(s+1)}, u_1, \ldots, u_{i-1}):u_i\big) \leq {\rm reg}(I(G)^{(s+1)}:u_i).
\end{array} \tag{1} \label{1}
\]

For every integer $i$ with $0\leq i\leq m$, set $I_i:=(I(G)^{(s+1)}, u_1, \ldots, u_i)$. In particular, $I_0=I(G)^{(s+1)}$ and $I_m=I(G)^{(s+1)}+I(G)^s$. Consider the exact sequence
$$0\rightarrow S/(I_{i-1}:u_i)(-2s)\rightarrow S/I_{i-1}\rightarrow S/I_i\rightarrow 0,$$
for every $1\leq i\leq m$. It follows that$${\rm reg}(I_{i-1})\leq \max \big\{{\rm reg}(I_{i-1}:u_i)+2s, {\rm reg}(I_i)\big\}.$$Therefore,
\begin{align*}
& {\rm reg}(I(G)^{(s+1)})={\rm reg}(I_0)\leq \max\big\{{\rm reg}(I_{i-1}:u_i)+2s, 1\leq i\leq m, {\rm reg}(I_m)\big\}\\ & =\max\big\{{\rm reg}(I_{i-1}:u_i)+2s, 1\leq i\leq m, {\rm reg}(I(G)^{(s+1)}+I(G)^s)\big\}.
\end{align*}
The assertion now follows from the inequality (\ref{1}).
\end{proof}

In order to use Lemma \ref{rfirst}, to bound the regularity of the $(s+1)$-th symbolic powers of an edge ideal, we first need to obtain information about the ideals of the form $(I(G)^{(s+1)}:u)$, where $u$ is an element of the set of minimal monomial generators of $I(G)^s$. The following lemma determines a generating set  for theses ideals, in the special case of $s=1$.

\begin{lem} \label{seccol}
Let $G$ be a graph and let $e=x_ix_j$ be an edge of $G$. Then
\begin{align*}
& \big(I(G)^{(2)}:e\big)=I(G)+\big(x_px_q: x_p\in N_G(x_i), x_q\in N_G(x_j), x_p\neq x_q\big)\\ & +\big(x_t: x_t\in N_G(x_i)\cap N_G(x_j)\big).
\end{align*}
\end{lem}

\begin{proof}
We first show that the right hand side is contained in the left hand side. Obviously, $I(G)\subseteq (I(G)^{(2)}:e)$. Let $x_p$ and $x_q$ be the vertices of $G$ with $x_p\in N_G(x_i)$ and $x_q\in N_G(x_j)$. Then$$x_px_qe=(x_px_i)(x_qx_j)\in I(G)^2\subseteq I(G)^{(2)}.$$Hence, $x_px_q\in (I(G)^{(2)}:e)$. Next, suppose $x_t$ is a vertex with $x_t\in N_G(x_i)\cap N_G(x_j)$. Assume that $T$ is the induced subgraph of $G$ on $\{x_i, x_j, x_t\}$. In particular, $T$ is a triangle. Then$$x_te=x_tx_ix_j\in I(T)^{(2)}\subseteq I(G)^{(2)}.$$Therefore, $x_t\in (I(G)^{(2)}:e)$.

We now prove the reverse inclusion. Let $u$ be a monomial in $(I(G)^{(2)}:e)$ and suppose$$u\notin I(G)+\big(x_t: x_t\in N_G(x_i)\cap N_G(x_j)\big).$$We show that$$u \in \big(x_px_q: x_p\in N_G(x_i), x_q\in N_G(x_j), x_p\neq x_q\big).$$

Let $A=\{x_{k_1}, \ldots, x_{k_{\ell}}\}$ be the set of variables dividing $u$. Since $u\notin I(G)$, it follows that $A$ is an independent subset of vertices of $G$. Assume that neither of $x_{k_1}, \ldots, x_{k_{\ell}}$ is adjacent to $x_i$. This implies that $A\cup\{x_i\}$ is an independent subset of vertices of $G$. In other words, $V(G)\setminus(A\cup\{x_i\})$ is a vertex cover of $G$. Therefore, there is a minimal vertex cover $C$ of $G$ which is contained in $V(G)\setminus(A\cup\{x_i\})$. Since $u \in(I(G)^{(2)}:e)$, it follows that $ue\in I(G)^{(2)}\subseteq \mathfrak{p}_C^2$. This is a contradiction, as $x_j$ is the only variable in $\mathfrak{p}_C$ which divides $ue$ and ${\rm deg}_{x_j}(ue)=1$. Hence, there is a variable, say $x_{k_a}$, dividing $u$, which is adjacent to $x_i$ in $G$. Similarly, there is a variable $x_{k_b}$ dividing $u$, which is adjacent to $x_j$. As$$u\notin\big(x_t: x_t\in N_G(x_i)\cap N_G(x_j)\big),$$we conclude that $x_{k_a}\neq x_{k_b}$. Thus,$$u \in \big(x_px_q: x_p\in N_G(x_i), x_q\in N_G(x_j), x_p\neq x_q\big),$$and we are done.
\end{proof}

In the following lemma, we study the ideals of the form $(I(G)^{(s+1)}:u)$, for some $u\in G(I(G)^s)$, and for any arbitrary integer $s\geq 1$.

\begin{lem} \label{col}
Assume that $G$ is a graph with edge set $E(G)=\{e_1, \ldots, e_r\}$, and let $s\geq 1$ be a positive integer. Then for any $s$-fold product $u=e_{i_1}\ldots e_{i_s}$, we have$$\big(I(G)^{(s+1)}:u\big)=\bigg(\big(I(G)^{(2)}:e_{i_1}\big)^{(s)}: e_{i_2}\ldots e_{i_s}\bigg).$$
\end{lem}

\begin{proof}
Let $C$ be a minimal vertex cover of $G$. Then for every integer $j$ with $1\leq j\leq s$, we have $|C\cap e_{i_j}|\geq 1$. If $|C\cap e_{i_j}|=1$, for all $j$, then$$\sum_{x_k\in C}{\rm deg}_{x_k}(u)=\sum_{x_k\in C}{\rm deg}_{x_k}(e_{i_1}\ldots e_{i_s})=s.$$Therefore, $(\mathfrak{p}_C^{s+1}:u)=\mathfrak{p}_C$. On the other hand, if $|C\cap e_{i_j}|=2$, for some integer $j$ with $1\leq j\leq s$, then$$\sum_{x_k\in C}{\rm deg}_{x_k}(u)\geq s+1.$$Hence, $u\in \mathfrak{p}_C^{s+1}$. In other words, $(\mathfrak{p}_C^{s+1}:u)=S$.

Let $\mathcal{A}$ denote the set of all minimal vertex covers $C$ of $G$ such that $|C\cap e_{i_j}|=1$, for any integer $j$ with $1\leq j\leq s$. It follows from the above argument that$$\big(I(G)^{(s+1)}:u\big)=\bigg(\bigcap_{C\in \mathcal{C}(G)}\mathfrak{p}_C^{s+1}\bigg): u=\bigcap_{C\in \mathcal{C}(G)}(\mathfrak{p}_C^{s+1}: u)=\bigcap_{C\in \mathcal{A}}\mathfrak{p}_C.$$

Let $\mathcal{A}'$ be the set of all minimal vertex covers $C$ of $G$ with $|C\cap e_{i_1}|=1$ and let $\mathcal{A}''$ be the set of all minimal vertex covers $C$ of $G$ such that $|C\cap e_{i_j}|=1$, for any integer $j$ with $2\leq j\leq s$. Obviously, $\mathcal{A}=\mathcal{A}'\cap\mathcal{A}''$. Using a similar argument as above, we have$$\big(I(G)^{(2)}:e_{i_1}\big)=\bigcap_{C\in \mathcal{A}'}\mathfrak{p}_C.$$Consequently,
\begin{align*}
& \bigg(\big(I(G)^{(2)}:e_{i_1}\big)^{(s)}: e_{i_2}\ldots e_{i_s}\bigg)=\bigcap_{C\in \mathcal{A}'}\big(\mathfrak{p}_C^s: e_{i_2}\ldots e_{i_s}\big)\\ & =\bigcap_{C\in \mathcal{A}'\cap \mathcal{A}''}\mathfrak{p}_C=\bigcap_{C\in \mathcal{A}}\mathfrak{p}_C=\big(I(G)^{(s+1)}:u\big).
\end{align*}
\end{proof}

Assume that $G$ is a graph and $e=x_ix_j$ is an edge of $G$. Let $G'$ be the graph which is obtained from $G$ by adding the edges of the form $x_px_q$ with $x_p\in N_G(x_i)$, $x_q\in N_G(x_j)$ and $x_p\neq x_q$. It is recently shown by Banerjee and Nevo \cite[Theorem 3.1]{bn} that ${\rm reg}(I(G'))\leq {\rm reg}(I(G))$. This following lemma can be deduced from this result. However, we provide an alternative proof.

\begin{lem} \label{base}
Let $G$ be a graph and $e$ be an edge of $G$. Then$${\rm reg}\big(I(G)^{(2)}:e\big)\leq {\rm reg}(I(G)).$$
\end{lem}

\begin{proof}
Suppose $x_i$ and $x_j$ are the endpoints of $e$. As $x_i$ and $x_j$ are adjacent in $G$, every minimal vertex cover of $G$ contains at least one of them. Let $\mathcal{A}_1$ denote the set of minimal vertex covers of $G$ which contain exactly one of $x_i$ and $x_j$. Also, assume $\mathcal{A}_2$ (resp. $\mathcal{A}_3$) is the set of minimal vertex covers $C\in \mathcal{A}_1$ with $x_i\in C$ (resp. $x_j\in C$). In particular, we have $\mathcal{A}_1=\mathcal{A}_2\cup \mathcal{A}_3$. Then
$$(I(G)^{(2)}:e)=\bigcap_{C\in \mathcal{A}_1}\mathfrak{p}_C=\big(\bigcap_{C\in\mathcal{A}_2}\mathfrak{p}_C\big)\cap\big(\bigcap_{C\in\mathcal{A}_3}\mathfrak{p}_C\big)=(I(G):x_j)\cap(I(G):x_i).$$

We know from from \cite[Lemma 4.2]{s3} that$${\rm reg}(I(G):x_i)\leq {\rm reg}(I(G)).$$Similarly,$${\rm reg}(I(G):x_j)\leq {\rm reg}(I(G)).$$

Let $H$ be the graph which is obtained from $G$ by deleting the vertices in $N_G(x_i)\cup N_G(x_i)$. Then$$(I(G):x_i)+(I(G):x_j)=I(H)+({\rm some \ variables}.)$$ As the disjoint union of $H$ and $e$ is an induced subgraph of $G$, it follows from \cite[Lemma 3.1]{ha} and \cite[Lemma 3.2]{ht} that ${\rm reg}(I(H))+\leq {\rm reg}(I(G))$. Therefore,$${\rm reg}\big((I(G):x_i)+(I(G):x_j)\big)={\rm reg}(I(H))\leq {\rm reg}(I(G))-1.$$

Consider the following exact sequence.
$$0\rightarrow S/(I(G)^{(2)}:e)\rightarrow S/(I(G):x_i)\oplus S/(I(G):x_j)\rightarrow S/\big((I(G):x_i)+(I(G):x_j)\big)\rightarrow 0.$$Using \cite[Corollary 18.7]{p'} and the above argument, we deduce that
\begin{align*}
& {\rm reg}\big(I(G)^{(2)}:e\big)\leq\max\{{\rm reg}(I(G):x_i), {\rm reg}(I(G):x_j), {\rm reg}\big((I(G):x_i)+(I(G):x_j)\big)\}\\ & \leq{\rm reg}(I(G)).
\end{align*}
\end{proof}

The following lemma extends Lemma \ref{base}.

\begin{lem} \label{lemreg}
Assume that $G$ is a graph with edge set $E(G)=\{e_1, \ldots, e_r\}$, and let $s\geq 1$ be a positive integer. Then for any $s$-fold product $u=e_{i_1}\ldots e_{i_s}$, we have$${\rm reg}\big(I(G)^{(s+1)}:u\big)\leq {\rm reg}(I(G)).$$
\end{lem}

\begin{proof}
We use induction on $s$. For $s=1$, the assertion follows from Lemma \ref{base}. Thus, assume that $s\geq 2$. Let $x$ and $y$ denote the endpoints of the edge $e_{i_1}$. Also, let $G'$ be the graph with edge ideal$$I(G')=I(G)+\big(x_px_q: x_p\in N_G(x_i), x_q\in N_G(x_j), x_p\neq x_q\big).$$Using Lemma \ref{seccol}, there exists a subset $A$ of variables with the property that$$\big(I(G)^{(2)}:e_{i_1}\big)=I(G')+(A).$$We know from Lemma \ref{col} that
\begin{align*}
& {\rm reg}\big(I(G)^{(s+1)}:u\big)={\rm reg}\bigg(\big(I(G)^{(2)}:e_{i_1}\big)^{(s)}: e_{i_2}\ldots e_{i_s}\bigg)\\ & ={\rm reg}\bigg(\big(I(G'), A\big)^{(s)}: e_{i_2}\ldots e_{i_s}\bigg)\\ & ={\rm reg}\bigg(\big(I(G'\setminus A), A\big)^{(s)}: e_{i_2}\ldots e_{i_s}\bigg).
\end{align*}
Without lose of generality, we may suppose that there exists an integer $\ell$ with the property that $e_{i_2}, \ldots, e_{i_{\ell}}$ have nonempty intersection with $A$, while $e_{i_{\ell+1}}, \ldots, e_{i_s}$ have no common vertex with $A$ (we set $\ell=1$ if neither of $e_{i_2}, \ldots, e_{i_s}$ is intersecting $A$). For every integer $k$ with $2\leq k\leq \ell$, let $x_k$ and $y_k$ be the endpoints of $e_{i_k}$, with $x_k\in A$. Then
\begin{align*}
& {\rm reg}\big(I(G)^{(s+1)}:u\big)={\rm reg}\bigg(\big(I(G'\setminus A), A\big)^{(s)}: e_{i_2}\ldots e_{i_s}\bigg)\\ & ={\rm reg}\bigg(\big(\big(\big(I(G'\setminus A), A\big)^{(s)}: x_2\ldots x_{\ell}\big): e_{i_{\ell+1}}\ldots e_{i_s}\big):y_2\ldots y_{\ell}\bigg)\\ & ={\rm reg}\bigg(\big(\big(I(G'\setminus A), A\big)^{(s-\ell+1)}: e_{i_{\ell+1}}\ldots e_{i_s}\big):y_2\ldots y_{\ell}\bigg),
\end{align*}
where the last equality follows from the assumption that $x_2, \ldots, x_{\ell}$ belong to $A$. We conclude from \cite[Lemma 4.2]{s3} and the above equalities that
\[
\begin{array}{rl}
{\rm reg}\big(I(G)^{(s+1)}:u\big)\leq {\rm reg}\bigg(\big(I(G'\setminus A), A\big)^{(s-\ell+1)}: e_{i_{\ell+1}}\ldots e_{i_s}\bigg).
\end{array} \tag{2} \label{2}
\]
By the choice of $\ell$, we know that $e_{i_{\ell+1}}, \ldots, e_{i_s}$ are edges of $G\setminus A$. Thus, $e_{i_{\ell+1}}\ldots e_{i_s}$ belongs to $I(G'\setminus A)^{(s-\ell)}$. Hence, for every variable $z\in A$, we have$$ze_{i_{\ell+1}}\ldots e_{i_s}\in \big(I(G'\setminus A), A\big)^{(s-\ell+1)}.$$In other words,$$z\in \bigg(\big(I(G'\setminus A), A\big)^{(s-\ell+1)}: e_{i_{\ell+1}}\ldots e_{i_s}\bigg).$$Therefore,$$\bigg(\big(I(G'\setminus A), A\big)^{(s-\ell+1)}: e_{i_{\ell+1}}\ldots e_{i_s}\bigg)=\bigg(\big(I(G'\setminus A)\big)^{(s-\ell+1)}: e_{i_{\ell+1}}\ldots e_{i_s}\bigg)+(A).$$It follows from inequality (\ref{2}) and the above equality that
\begin{align*}
& {\rm reg}\big(I(G)^{(s+1)}:u\big)\leq {\rm reg}\bigg(\big(\big(I(G'\setminus A)\big)^{(s-\ell+1)}: e_{i_{\ell+1}}\ldots e_{i_s}\big)+(A)\bigg)\\ & ={\rm reg}\bigg(\big(I(G'\setminus A)\big)^{(s-\ell+1)}: e_{i_{\ell+1}}\ldots e_{i_s}\bigg).
\end{align*}
As $s-\ell< s$, we deduce from the induction hypothesis that
\begin{align*}
{\rm reg}\big(I(G)^{(s+1)}:u\big)\leq {\rm reg}\bigg(\big(I(G'\setminus A)\big)^{(s-\ell+1)}: e_{i_{\ell+1}}\ldots e_{i_s}\bigg) \leq {\rm reg}(I(G'\setminus A)).
\end{align*}
Thus, using \cite[Lemma 3.1]{ha} and  Lemma \ref{base}, we have$${\rm reg}\big(I(G)^{(s+1)}:u\big)\leq {\rm reg}(I(G'\setminus A))\leq {\rm reg}(I(G'))\leq {\rm reg}(I(G)).$$
\end{proof}

We are now ready to prove the first main result of this paper.

\begin{thm} \label{1main}
For any graph $G$ and for every integer $s\geq 1$, we have$${\rm reg}(I(G)^{(s+1)})\leq \max\bigg\{{\rm reg}(I(G))+2s, {\rm reg}\big(I(G)^{(s+1)}+I(G)^s\big)\bigg\}.$$
\end{thm}

\begin{proof}
The assertion follows from Lemmata \ref{rfirst} and \ref{lemreg}.
\end{proof}

In view of Theorem \ref{1main}, in order to bound the regularity of $I(G)^{(s+1)}$, we need to estimate the regularity of the ideal $I(G)^{(s+1)}+I(G)^s$. Let $G$ be a graph and $k\geq 1$ be an integer. Rinaldo, Terai and Yoshida \cite[Lemma 3.10]{rty} proved that if $G$ has no odd cycle of length at most $2k-1$, then for every integer $s\leq k$, the equality $I(G)^{(s)}=I(G)^s$ holds. Using this result, we prove the following proposition, which states that the ideal $I(G)^{(s+1)}+I(G)^s$ has a nice description, if $s$ is sufficiently small.

\begin{prop} \label{cont}
Let $G$ be a graph and let $s\geq 1$ be an integer with the property that $G$ has no odd cycle of length at most $2s-3$. Then $I(G)^{(s+1)}\subseteq I(G)^s$.
\end{prop}

\begin{proof}
Let $u$ be a monomial in $I(G)^{(s+1)}$. Since $I(G)^{(s+1)}\subseteq I(G)$, there is an edge $e:=xy\in E(G)$ which divides $u$. Set $v:=u/(xy)$. Then$$v\in (I(G)^{(s+1)}:xy)\subseteq I(G)^{(s-1)}.$$As $G$ has no odd cycle of length at most $2s-3$, it follows from \cite[Lemma 3.10]{rty} that $I(G)^{(s-1)}=I(G)^{s-1}$. Therefore, $v\in I(G)^{s-1}$. This implies that $u=(xy)v=ev$ belongs to $I(G)^s$.
\end{proof}

As an immediate consequence of Theorem \ref{1main} and Proposition \ref{cont}, we obtain the following theorem.

\begin{thm} \label{2main}
Let $G$ be a graph and let $s\geq 1$ be an integer with the property that $G$ has no odd cycle of length at most $2s-3$. Then $${\rm reg}(I(G)^{(s+1)})\leq \max\big\{{\rm reg}(I(G))+2s, {\rm reg}(I(G)^s)\big\}.$$
\end{thm}

We can now prove that Conjecture \ref{conj3} is true for $s=2,3$.

\begin{cor} \label{twth}
For any graph $G$ we have,
\begin{itemize}
\item[(i)] ${\rm reg}(I(G)^{(2)})\leq {\rm reg}(I(G))+2$, and
\item[(ii)] ${\rm reg}(I(G)^{(3)})\leq {\rm reg}(I(G))+4$.
\end{itemize}
\end{cor}

\begin{proof}
(i) immediately follows from Theorem \ref{2main}, while (ii) follows from Theorem \ref{2main} and \cite[Theorem 1.1]{bn}.
\end{proof}

Using the above mentioned result of Rinaldo, Terai and Yoshida \cite[Lemma 3.10]{rty}, we are also able to prove the following corollary which extends \cite[Theorem 1.1(ii)]{bn}.

\begin{cor} \label{regord}
Let $G$ be a graph and let $k\geq 1$ be an integer with the property that $G$ has no odd cycle of length at most $2k-1$. Then for every integer $s\leq k$, we have$${\rm reg}(I(G)^s)\leq 2s+{\rm reg}(I(G))-2.$$
\end{cor}

\begin{proof}
We use induction on $s$. For $s=1$, there is nothing to prove. Thus, assume that $s\geq 2$. It follows from the induction hypothesis that$${\rm reg}(I(G)^{s-1})\leq 2s+{\rm reg}(I(G))-4.$$Hence, we conclude from \cite[Lemma 3.10]{rty} and Theorem \ref{2main} that
\begin{align*}
& {\rm reg}(I(G)^s)={\rm reg}(I(G)^{(s)})\leq \max\big\{{\rm reg}(I(G))+2s-2, {\rm reg}(I(G)^{s-1})\big\}\\ & = 2s+{\rm reg}(I(G))-2.
\end{align*}
\end{proof}

Let $G$ be a graph and $k\geq 1$ be an integer. Assume that $G$ has no odd cycle of length at most $2k-1$. It follows from \cite[Lemma 3.10]{rty} and Corollary \ref{regord} that for every integer $s\leq k$, we have ${\rm reg}(I(G)^{(s)})\leq 2s+{\rm reg}(I(G))-2$. The following corollary shows that the same inequality holds for $s=k+1$, too.

\begin{cor} \label{resycy}
Let $G$ be a graph and let $k\geq 1$ be an integer with the property that $G$ has no odd cycle of length at most $2k-1$. Then for every integer $s\leq k+1$, we have$${\rm reg}(I(G)^{(s)})\leq 2s+{\rm reg}(I(G))-2.$$
\end{cor}

\begin{proof}
The assertion follows from Theorem \ref{2main} and Corollary \ref{regord}.
\end{proof}

%%%%%%%%%%%%%%%%%%%%%%%%%%%%%%%%%%%%%%%%%%%%%%%%%%%%%%%%%%%%%%%%%%%%%%%%%%

\section{Graphs with chordal complement} \label{sec4}

Let $G$ be a graph with the property that the complementary graph $\overline{G}$ is chordal. As we mentioned in the introduction, for every integer $s\geq 1$, we have ${\rm reg}(I(G)^s)=2s$. However, the regularity of symbolic powers of $I(G)$ is not in general known. It immediately follows from Corollary \ref{twth} that ${\rm reg}(I(G)^{(s)})=2s$, for $s=2, 3$. In Theorem \ref{fococh}, we show that the same equality is also true for $s=4$. The proof of this result is based on Theorem \ref{1main}. To use this theorem, we need to estimate the regularity of the ideal $I(G)^{(4)}+I(G)^3$. The following lemma is needed for this estimation.

\begin{lem} \label{fouthr}
Let $G$ be a gap-free graph and let $u$ be a monomial in the set of minimal monomial generators of $I(G)^2$. Assume that $X_0$ is the set of variables belonging to $(I(G)^{(4)}: u)$ (if any). Then $$(I(G)^{(4)}:u)+(I(G)^3:u)=(I(G)^3:u)+(X_0).$$
\end{lem}

\begin{proof}
The inclusion $"\supseteq"$ is obvious. Hence, we prove the reverse inclusion.

As $u$ is in the set of minimal monomial generators of $I(G)^2$, we may write $u=e_1e_2$, for some edges $e_1, e_2\in E(G)$. Let $v$ be a monomial in  $(I(G)^{(4)}:u)$. If $v\in I(G)$, then $v$ belongs to $(I(G)^3:u)$ and we are done. Thus, assume that $v\notin I(G)$. Let $A$ denote the set of variables dividing $v$. As $v\notin I(G)$, it follows that $A$ is an independent subset of vertices of $G$. We consider the following cases.

\vspace{0.3cm}
{\bf Case 1.} Suppose $e_1=e_2$. Let $x$ and $y$ denote the endpoint of the edge $e_1=e_2$. If no vertex in $A$ is adjacent to $x$, then $A\cup \{x\}$ is an independent subset of vertices of $G$ (in particular, $y\notin A$). This implies that $V(G)\setminus (A\cup \{x\})$ is a vertex cover of $G$. Let $C$ be a minimal vertex of $G$ which is contained in $V(G)\setminus (A\cup \{x\})$. Note that$$C\cap (A\cup \{x, y\})=\{y\}.$$This is a contradiction, because$$vx^2y^2\in I(G)^{(4)}\subseteq \mathfrak{p}_C^4,$$ which implies that $2={\rm deg}_y(vx^2y^2)\geq 4$. Consequently, there is a vertex $z\in A$ which is adjacent to $x$.

Set $v':=v/z$ and denote by $A'$ the set of variables dividing $v'$. Assume that no vertex in $A'$ is adjacent to $y$ (in particular $x\notin A'$). Then $A'\cup \{y\}$ is an independent subset of vertices of $G$. In other words, $V(G)\setminus (A'\cup \{y\})$ is a vertex cover of $G$. Let $C'$ be a minimal vertex of $G$ which is contained in $V(G)\setminus (A'\cup \{y\})$. Note that$$\{x\}\subseteq C'\cap (A\cup \{x, y\})\subseteq\{z,x\}.$$If $C'\cap (A\cup \{x, y\})=\{x\}$, then it follows from$$vx^2y^2\in I(G)^{(4)}\subseteq \mathfrak{p}_{C'}^4$$that $2={\rm deg}_x(vx^2y^2)\geq 4$, which is a contradiction. Therefore,$$C'\cap (A\cup \{x, y\})=\{z,x\}.$$ This implies that $z\notin A'\cup\{y\}$. Consequently, ${\rm deg}_z(vx^2y^2)=1$. Since, ${\rm deg}_x(vx^2y^2)=2$, we conclude that$${\rm deg}_z(vx^2y^2)+{\rm deg}_x(vx^2y^2)=3,$$which contradicts the inclusion $vx^2y^2\in \mathfrak{p}_{C'}^4$. Hence, there is a vertex $z'\in A'$ which is adjacent to $y$.

Note that$$zz'u=zz'x^2y^2=(zx)(z'y)(xy)\in I(G)^3,$$which implies that $zz'\in (I(G)^3:u)$. Since $v$ is divisible by $zz'$, we deduce that $v\in (I(G)^3:u)$.

\vspace{0.3cm}
{\bf Case 2.} Suppose $e_1$ and $e_2$ are distinct edges which have a common vertex. Let $x$ and $y$ (resp. $x$ and $y'$) denote the endpoints of $e_1$ (resp. $e_2$). Let $w$ be an arbitrary vertex in the set $\{x, y, y'\}$. Using a similar argument as in Case 1, there is a vertex $z_w\in A$ which is adjacent to $w$. Thus, there are vertices $z_x, z_y, z_{y'}\in A$ with $xz_x, yz_y, y'z_{y'}\in I(G)$. If $z_x\neq z_y$, then$$z_xz_yu=z_xz_yx^2yy'=(xz_x)(yz_y)(xy')\in I(G)^3.$$Consequently, $z_xz_y\in (I(G)^3:u)$, and since $v$ is divisible by $z_xz_y$, it follows that $v\in (I(G)^3:u)$. Similarly, if $z_x\neq z_{y'}$, it again follows that $v\in (I(G)^3:u)$. Thus, assume that $z_x=z_y=z_{y'}$.

Set $v':=v/z_x$ and denote by $A'$ the set of variables dividing $v'$. Assume that no vertex in $A'$ is adjacent to $x$ (in particular $y, y', z_x\notin A'$). Then $A'\cup \{x\}$ is an independent subset of vertices of $G$. This implies that $V(G)\setminus (A'\cup \{x\})$ is a vertex cover of $G$. Let $C$ be a minimal vertex of $G$ which is contained in $V(G)\setminus (A'\cup \{x\})$. It follows that$$C\cap (A\cup \{x, y, y'\})\subseteq\{y, y', z_x\}.$$ Since, $y, y', z_x\notin A'$, we have$${\rm deg}_y(vx^2yy')={\rm deg}_{y'}(vx^2yy')={\rm deg}_{z_x}(vx^2yy')=1.$$Hence, $vx^2yy'=vu\notin \mathfrak{p}_C^4$ and this is a contradiction. Therefore, there is a vertex $z'_x\in A'$ which is a adjacent to $x$.

We recall that $z_x=z_y=z_{y'}$. Note that$$z_xz_x'u=z_xz_x'x^2yy'=(z_x'x)(z_xy)(xy')\in I(G)^3.$$ Hence, $z_xz_x'$ belongs to $(I(G)^3:u)$. As $z_x'$ divides $v'$, we deduce that $z_xz_x'$ divides $v$ and thus, $v\in (I(G)^3:u)$.

\vspace{0.3cm}
{\bf Case 3.} Suppose $e_1$ and $e_2$ are disjoint edges of $G$. Let $x$ and $y$ (resp. $x'$ and $y'$) denote the endpoints of $e_1$ (resp. $e_2$). As $G$ is a gap-free graph, $e_1$ and $e_2$ can not form a gap in $G$. Without lose of generality, assume that $x$ and $x'$ are adjacent in $G$. Let $w$ be an arbitrary vertex in the set $\{x, y, x', y'\}$. Using a similar argument as in Case 1, there is a vertex $z_w\in A$ which is adjacent to $w$. Thus, there are vertices $z_x, z_y, z_{x'}, z_{y'}\in A$ with $xz_x, yz_y, x'z_{x'}, y'z_{y'}\in I(G)$. If $z_x\neq z_y$, then$$z_xz_yu=z_xz_yxyx'y'=(xz_x)(yz_y)(x'y')\in I(G)^3.$$Consequently, $z_xz_y\in (I(G)^3:u)$. Since $v$ is divisible by $z_xz_y$, it follows that $v\in (I(G)^3:u)$. Hence, suppose $z_x=z_y$. Similarly, if $z_{x'}\neq z_{y'}$, it again follows that $v\in (I(G)^3:u)$. Thus, assume that $z_{x'}=z_{y'}$. If $z_x\neq z_{x'}$, then it follows from $xx'\in I(G)$ and the equalities $z_x=z_y$ and $z_{x'}=z_{y'}$ that$$z_xz_{x'}u=z_xz_{x'}xyx'y'=(yz_y)(y'z_{y'})(xx')\in I(G)^3.$$Consequently, $z_xz_{x'}\in (I(G)^3:u)$. Since $v$ is divisible by $z_xz_{x'}$, it follows that $v\in (I(G)^3:u)$. Therefore, we assume that $z_y=z_x=z_{x'}=z_{y'}$.

If the induced subgraph of $G$ on $\{x, y, x', y'\}$ is a complete graph, then as $z_x$ is adjacent to all of these vertices, we conclude that the induced subgraph of $G$ on $\{x, y, x', y', z_x\}$ is the complete graph $K_5$. It follows that for every minimal vertex cover $C$ of $G$, we must have$$|C\cap \{x, y, x', y', z_x\}|\geq 4.$$This implies that$$z_xu=z_xxyx'y'\in \mathfrak{p}_C^4,$$for every $C\in \mathcal{C}(G)$. Thus, $z_xu\in I(G)^{(4)}$. Consequently, $z_x\in X_0$, and therefore, we have $u\in (X_0)$. Hence, assume that there are two vertices, say $a$ and $b$, in $\{x, y, x', y'\}$, such that $ab\notin I(G)$. Let $c$ and $d$ denote the two vertices of $\{x, y, x', y'\}\setminus\{a, b\}$. As $ab\notin I(G)$, we have either $ac, bd\in I(G)$, or $ad, bc\in I(G)$.

Set $v':=v/z_x$ and denote by $A'$ the set of variables dividing $v'$. Assume that $A'\cup\{a, b\}$ is an independent subset of vertices of $G$ (in particular, $c, d, z_x\notin A'$). Then $V(G)\setminus (A'\cup \{a, b\})$ is a vertex cover of $G$. Let $C'$ be a minimal vertex of $G$ which is contained in $V(G)\setminus (A'\cup \{a,b\})$. It follows that$$C'\cap (A\cup \{x, y, x', y'\})\subseteq\{c, d, z_x\}.$$ Since, $c, d, z_x\notin A'$, we have$${\rm deg}_c(vxyx'y')={\rm deg}_d(vxyx'y')={\rm deg}_{z_x}(vxyx'y')=1.$$Hence, $vxyx'y'=vu\notin \mathfrak{p}_{C'}^4$ and this is a contradiction. Therefore, $A'\cup\{a, b\}$ is not an independent subset of vertices of $G$. As, $A'\subseteq A$ is an independent subset of vertices of $G$, we deduce that there is a vertex $z\in A'$ with $az\in I(G)$, or $bz\in I(G)$. Without lose of generality suppose $az\in I(G)$. As we mentioned above, either $bc\in I(G)$, or $bd\in I(G)$. If $bc\in I(G)$, then$$zz_xu=zz_xxyx'y'=(az)(dz_x)(bc)\in I(G)^3.$$This implies that $zz_x\in (I(G)^3:u)$. As $v$ is divisible by $zz_x$, we conclude that $v\in (I(G)^3:u)$. Similarly, if $bd\in I(G)$, then $$zz_xu=(az)(cz_x)(bd)\in I(G)^3,$$which again implies that $v\in (I(G)^3:u)$.
\end{proof}

The following example shows that the assertion of Lemma \ref{fouthr} is not true if $G$ has a gap.

\begin{exmp}
Let $G$ be the graph with vertex set $V(G)=\{x_1, \ldots, x_7\}$ and edge set$$E(G)=\{x_1x_2, x_1x_3, x_2x_3, x_3x_4, x_4x_5, x_5x_6, x_5x_7, x_6x_7\}.$$Set $u=(x_1x_2)(x_6x_7)$. Then $x_3x_5\in (I(G)^{(4)}:u)$, while $x_3x_5\notin (I(G)^3:u)+(X_0)$, where $X_0$ is the subset of variable defined in Lemma \ref{fouthr}. Note that the edges $x_1x_2$ and $x_6x_7$ form a gap in $G$.
\end{exmp}

We are now ready to prove the main result of this section.

\begin{thm} \label{fococh}
Let $G$ be a graph such that the complementary graph $\overline{G}$ is chordal. Then$${\rm reg}\big(I(G)^{(s)}\big)=2s,$$for every $s\in \{2, 3, 4\}$.
\end{thm}

\begin{proof}
We know from Fr${\rm \ddot{o}}$berg's result \cite[Theorem 1]{f} that ${\rm reg}(I(G))=2$. For $s=2,3$, we conclude from Corollary \ref{twth} that$$2s\leq {\rm reg}(I(G)^{(s)})\leq {\rm reg}(I(G))+2s-2=2s.$$Thus, the assertion follows in these cases.

We now assume that $s=4$. By Theorem \ref{1main},$$8\leq{\rm reg}(I(G)^{(4)})\leq \max\big\{8, {\rm reg}\big(I(G)^{(4)}+I(G)^3\big)\big\}.$$Thus, it is enough to show that$${\rm reg}\big(I(G)^{(4)}+I(G)^3\big)\leq 8.$$Indeed, we prove$${\rm reg}\big(I(G)^{(4)}+I(G)^3\big)\leq 6.$$

Let $G(I(G)^2)=\{u_1, \ldots, u_m\}$ denote the set of minimal monomial generators of $I(G)^2$. Using \cite[Theorem 4.12]{b}, we may assume that for every pair of integers $1\leq j< i\leq m$, one of the following conditions hold.
\begin{itemize}
\item [(i)] $(u_j:u_i) \subseteq (I(G)^3:u_i)$; or
\item [(ii)] there exists an integer $k\leq i-1$ such that $(u_k:u_i)$ is generated by a subset of variables, and $(u_j:u_i)\subseteq (u_k:u_i)$.
\end{itemize}
Consequently, for every integer $i\geq 2$, we have
\[
\begin{array}{rl}
\big((I(G)^3, u_1, \ldots, u_{i-1}):u_i\big)=(I(G)^3:u_i)+({\rm some \ variables}).
\end{array} \tag{3} \label{3}
\]

For every integer $i$ with $0\leq i\leq m$, set$$I_i:=(I(G)^{(4)}+I(G)^3, u_1, \ldots, u_i).$$In particular, $I_0=I(G)^{(4)}+I(G)^3$. Consider the exact sequence
$$0\rightarrow S/(I_{i-1}:u_i)(-4)\rightarrow S/I_{i-1}\rightarrow S/I_i\rightarrow 0,$$
for every $1\leq i\leq m$. It follows that$${\rm reg}(I_{i-1})\leq \max \big\{{\rm reg}(I_{i-1}:u_i)+4, {\rm reg}(I_i)\big\}.$$Therefore,
\[
\begin{array}{rl}
{\rm reg}(I(G)^{(4)}+I(G)^3)={\rm reg}(I_0)\leq \max\big\{{\rm reg}(I_{i-1}:u_i)+4, 1\leq i\leq m, {\rm reg}(I_m)\big\}.
\end{array} \tag{4} \label{4}
\]
Note that $I_m=I(G)^{(4)}+I(G)^2$. On the other hand, we know from Proposition \ref{cont} that$$I(G)^{(4)}\subseteq I(G)^{(3)}\subseteq I(G)^2.$$Thus, $I_m=I(G)^2$. Consequently, ${\rm reg}(I_m)=4$. Hence, inequality (\ref{4}) implies that
\[
\begin{array}{rl}
{\rm reg}(I(G)^{(4)}+I(G)^3)\leq \max\big\{{\rm reg}(I_{i-1}:u_i)+4, 1\leq i\leq m\big\}.
\end{array} \tag{5} \label{5}
\]
As the complementary graph $\overline{G}$ has no induced $4$-cycle, it follows that $G$ is a gap-free graph. Therefore, using Lemma \ref{fouthr} and equality (\ref{3}), we conclude that for every integer $i$ with $1\leq i\leq m$,
\begin{align*}
& (I_{i-1}:u_i)=\big((I(G)^{(4)}+I(G)^3, u_1, \ldots, u_{i-1}):u_i\big)\\ & =\big((I(G)^{(4)}+I(G)^3):u_i\big)+\big((I(G)^3, u_1, \ldots, u_{i-1}):u_i\big)\\ & =(I(G)^3:u_i)+({\rm some \ variables}).
\end{align*}
It then follows from \cite[Lemma 2.10]{b} that
\[
\begin{array}{rl}
{\rm reg}(I_{i-1}:u_i) \leq {\rm reg}(I(G)^3:u_i).
\end{array} \tag{6} \label{6}
\]
We know from the proof of \cite[Theorem 6.16]{b} that ${\rm reg}(I(G)^3:u_i)=2$. Hence, using inequality (\ref{6}), we have ${\rm reg}(I_{i-1}:u_i)\leq 2$. Finally, it follows from inequality (\ref{5}) that$${\rm reg}(I(G)^{(4)}+I(G)^3)\leq 6,$$and this completes the proof.
\end{proof}

%%%%%%%%%%%%%%%%%%%%%%%%%%%%%%%%%%%%%%%%%%%%%%%%%%%%%%%%%%%%%%%%%%%%%%%%%%

%\section*{Acknowledgment}

%%%%%%%%%%%%%%%%%%%%%%%%%%%%%%%%%%%%%%%%%%%%%%%%%%%%%%%%%%%%%%%%%%%%%%%%%%

%%%%%%%%%%%%%%%%%%%%%%%%%%%%%%%%%%%%%%%%%%%%%%%%%%%%%%%%%%%%%%%%%%%%%%%%%%

\end{document}